\documentclass{amsart}
\usepackage{amsfonts,amssymb,amsmath,amsthm}
\usepackage{url}
\usepackage{enumerate}
\usepackage{graphicx,amssymb,mathrsfs,amsmath,latexsym,amsfonts,amsthm}
\usepackage{amssymb,amsmath}

\usepackage{latexsym,bm}

\usepackage{dsfont}
\usepackage{amsmath,amssymb,amsthm}

\usepackage{latexsym}

\usepackage{amssymb}

\usepackage{stmaryrd}

\usepackage{float}

\usepackage{psfrag}

\usepackage{xcolor}

\usepackage{enumerate}

\usepackage{manfnt,mathrsfs}

\usepackage{epstopdf}
\newtheorem{theorem}{Theorem}[section]
\newtheorem{lemma}[theorem]{Lemma}

\theoremstyle{definition}

\numberwithin{equation}{section}

\newcommand{\A}{\mathcal{A}}

\newcommand{\V}{\mathcal{V}}

\def\tu{{\rm Tur{\'a}n\,}}

\author{Zejun Huang, Zhenhua Lyu}
\address{Zejun Huang,
Institute of Mathmatics\\
Hunan University\\
Changsha  410082, P.R. China.}
\address{
	Zhenhua Lyu, College of Mathematics and Econometrics, Hunan University, Changsha  410082, P.R. China.  }
\email{mathzejun@gmail.com (Huang); lyuzhh@outlook.com (Lyu)}
\thanks{ }

\begin{document}

\title[Extremal digraphs avoiding an orientation of
$C_4$]{Extremal digraphs avoiding an orientation of
	$C_4$}

\begin{abstract}
	Let $P_{2,2}$ be the  orientation of $C_4$ which  consists of two 2-paths with the same initial and  terminal vertices.
In this paper, we determine the maximum size of $P_{2,2}$-free digraphs of order $n$ as well as the extremal digraphs attaining the maximum size when $n\ge 13$.
\end{abstract}
 \maketitle
	
	\section{Introduction}
	
	Digraphs in this paper are strict, i.e., they do not allow loops or parallel arcs. For digraphs, we abbreviate directed
	paths and directed cycles as paths and cycles, respectively. The number of vertices in a digraph is called its
	{\it order} and the number of arcs its {\it size}.  Given two digraphs $D$ and $H$, we say $D$ is {\it $H$-free} if $D$ does not contain an $H$ as its subgraph. Denote by   $K_r$ (or $\overrightarrow{K}_r$) the complete graph (or digraph) of order $r$ and $C_r$ (or $\overrightarrow{C}_r$)  the cycle (or directed cycle)  with $r$ vertices.

	\tu problem  is a hot topic in   graph theory. It concerns the possible largest number of edges in graphs without given subgraphs and the extremal graphs achieving that maximum number of edges. It is initiated by Tur{\'a}n's generalization of Mantel's theorem \cite{PT,PT2}, which determined the maximum size of $K_r$-free graphs on $n$ vertices and the unique extremal graph  attaining that
	maximum size. Most results on classical \tu problems concern undirected graphs and only a few \tu problems
	on digraphs have been investigated; see \cite{BB,BES,BES2,BH,BS,HL,HZ,AS}. In this paper we consider a \tu problem on digraphs.

	A natural \tu problem on digraphs is determining the maximum size of a $\overrightarrow{K}_r$-free digraph of a given order, which has been solved in \cite{JM}. Brown and Harary \cite{BH}  determined the precise extremal sizes and extremal digraphs for digraphs avoiding a tournament, which is an orientation of a complete graph. They also studied digraphs avoiding a direct sum of two tournaments, or a digraph on at most 4 vertices where any two vertices are joined by at least one arc.  By using dense matrices, asymptotic results on  extremal digraphs avoiding a family of digraphs were presented in \cite{BES,BES2,BES3}. In \cite{HDT1,HDT2} the authors determined the extremal sizes of  $\overrightarrow{C}_2$-free digraphs avoiding $k$ directed paths with the same initial vertex and terminal vertex for $k=2,3$.   Maurer, Rabinovitch and Trotter \cite{MRT} studied the extremal  transitive $\overrightarrow{C}_2$-free digraphs which  contain at most one directed path from $x$ to $y$  for any two distinct vertices $x,y.$ In \cite{ HL,HZ, HW}, the authors studied the extremal digraphs which have no distinct walks of a given length $k$ with the same initial vertex and the same terminal vertex.

	Notice that the $k$-cycle is a  generalization  of the triangle  when we view a triangle
	as a 3-cycle in  undirected graphs. Another generalization of Mantel's Theorem is the
	\tu problem  for $C_k$-free  graphs. However, it is very difficult to determine the exact maximum size of $C_k$-free graphs of a given order even for $k=4$; see \cite{DDR, PE,FK,ZF, FS,GHO,  IR, MS, TT, KZ}.

	Following the direction of Brown and Harary, it is interesting to investigate the extremal problem for digraphs avoiding an specific orientation of a cycle.
 Among all the orientations of $C_k$,  the directed cycle
	$\overrightarrow{C}_k$ is one of
	the most natural orientations, whose \tu number is difficult to determine and we leave this problem for future research.
	
	When $k$ is even, another natural orientation of $C_k$ is  the  union of two directed $k/2$-paths, which share the same initial vertex.  We will consider the case $k=4$ in this paper. Let $P_{2,2}$  be the following orientation of $C_4$.
	\begin{figure}[H]
		\centering
		\includegraphics[width=0.8in]{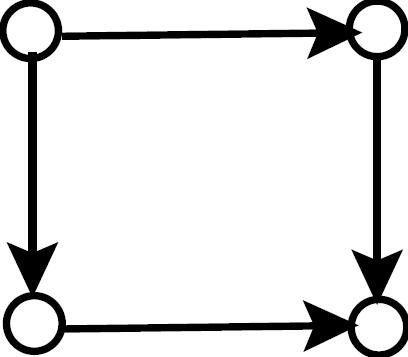}\hspace{0.8cm}\\
		$P_{2,2}$
	\end{figure}

	Let $ex(n)$ be  the maximum size  of $P_{2,2}$-free digraphs of order $n$ and $EX(n)$ be the set of $P_{2,2}$-free  digraphs attaining $ex(n)$.  In this paper, we solve the following  problem  for   $n\ge 13$.

	{\it {\bf Problem.} Let $n$ be a positive integer. Determine $ex(n)$ and $EX(n)$.}

	\section{Main results}

	In order to present our results, we need the following notations and definitions.
	
	Denote by  $D=(\mathcal{V},\mathcal{A})$ a digraph with  vertex set $\mathcal{V}$ and arc set $\mathcal{A}$. For a subset $X\subset \mathcal{V}$, we denote by $D(X)$ the subdigraph of $D$ induced by $X$.  For convenience, if $X=\{x\}$ is a singleton, it will be abbreviated as $x$.

	For $u,w\in \V$, if there is an arc from $u$ to $w$, then we say $w$ is a {\it successor} of $u$, and $u$ is a  {\it predecessor} of $w$. The notation  $(u,w)$ or $u\rightarrow w$ means there is an arc from $u$ to $w$;  $u\nrightarrow w$ means there exists no arc from $u$ to $w$;   $u\leftrightarrow w$ means both $u \rightarrow w$ and $w\rightarrow u$. For $S,T\subset \V$,  the notation $S\rightarrow T$ means there exists a vertex $x\in S$ such that $x\rightarrow y$ for any vertex $y\in T$;  $S\nrightarrow T$  means there is no arc from $S$ to $T$.   If every vertex in $S$ has a unique successor in $T$ and each vertex in $T$ has a unique predecessor in $S$, we say $S$ {\it matches} $T$. Note that $S$ matching $T$ indicate  $|S|=|T|$.
	We denote by $\A(S,T)$ the set of arcs from $S$ to $T$, which will be abbreviate as $\A(S)$ when $S=T$.  The cardinality of $\A(S,T)$ is denoted by $e(S,T)$.

	For $W,S\subset \V$,  denote by $$N^+_{W}(u)=\{x\in W|(u,x)\in \A\},$$ $$ N^-_{W}(u)=\{x\in W|(x,u)\in \A\},$$and
	$$ N^+_{W}(S)=\bigcup\limits_{u\in S}N^+_{W}(u),$$which are   simplified as $N^+(u)$, $N^-(u)$ and $N^+(S)$ when $W=\V$.

	If two digraphs $G$ and $H$ has disjoint vertex sets, their union is called a {\it disjoint union}.
	
  If a digraph $D$ is acyclic and there is a vertex $u$ such that there is a unique directed path from $u$ to any other vertex, then we say $D$ is  an {\it arborescence} with root $u$.  If the maximum length of these paths is at most $r$, then we say $D$ is  an {\it $r$-arborescence} with root $u$. Moreover, if $D$ is a 1-arborescence, then we also say $D$ is an {\it out-star} with center $u$.
	
We will use $S(x,y)$ and $T(x,y)$ to denote the following digraphs, whose orders will be clear from the context. Note that
$S(x,y)$ is the union of a 2-cycle $x\leftrightarrow y$ and a out-star with center $y$; $T(x,y)$ is the union of a  2-cycle $x\leftrightarrow y$ and two 2-arborescences with roots $x$ and $y$.

	\begin{figure}[H]
		\centering
		\includegraphics[width=1.5in]{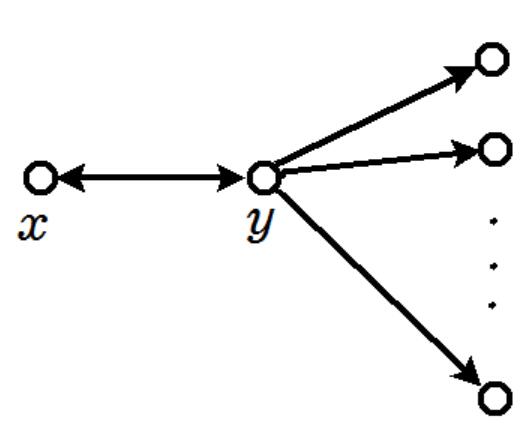}\hspace{0.8cm}
		\includegraphics[width=2.5in]{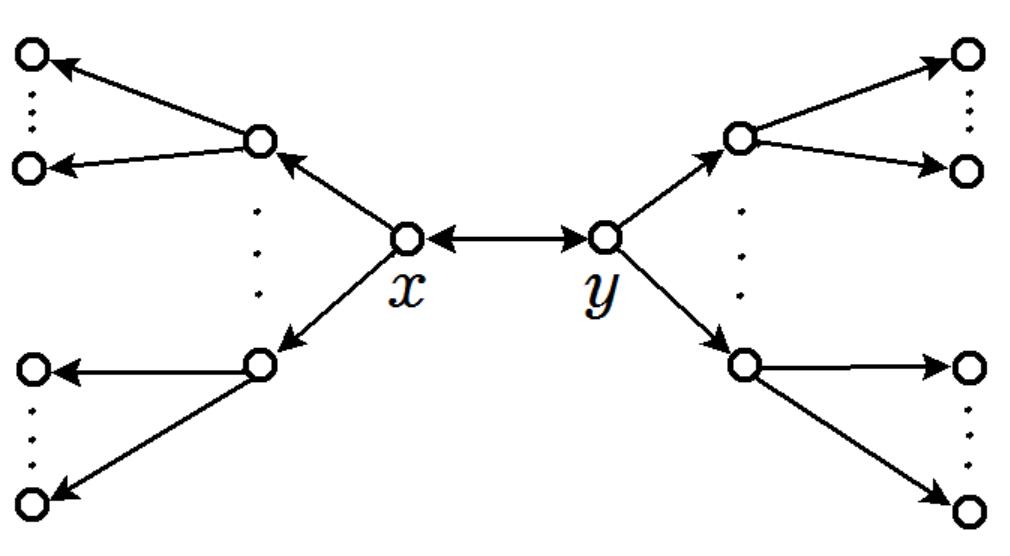}\hspace{0.8cm}\\	$S(x,y)$\hspace{4.4cm}$T(x,y)$
	\end{figure}

	Now we present the following nine classes of digraphs of order $n$. Each of these digraphs has  vertex partition $\V_1\cup \V_2$ with $|\V_1|=\lfloor\frac{n}{2}\rfloor+1$.

	\begin{figure}[H]
		\centering
		\includegraphics[width=1.6in]{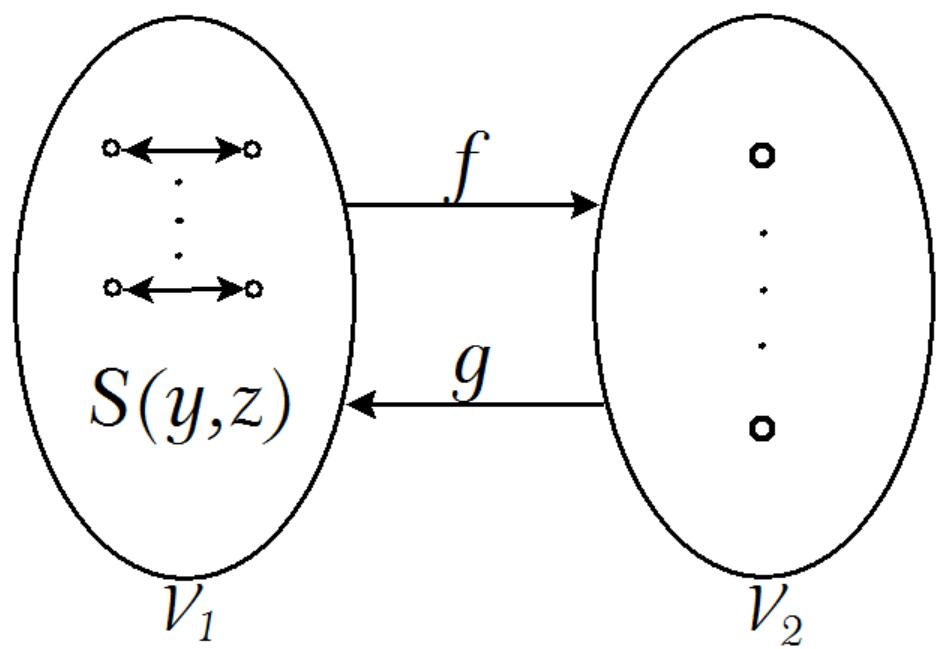}\hspace{0.1cm}
		\includegraphics[width=1.6in]{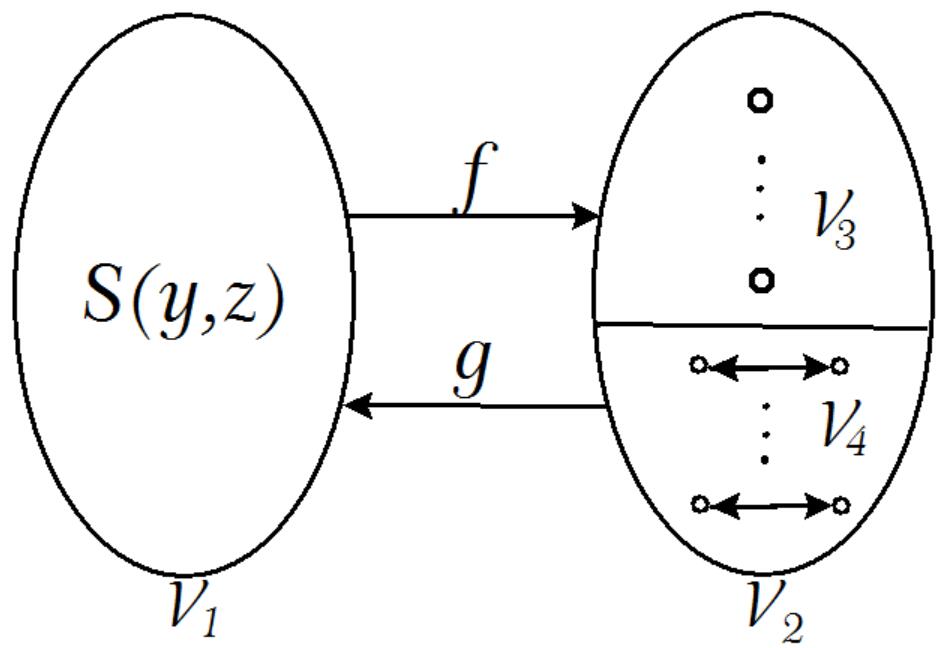}\hspace{0.1cm}
			\includegraphics[width=1.6in]{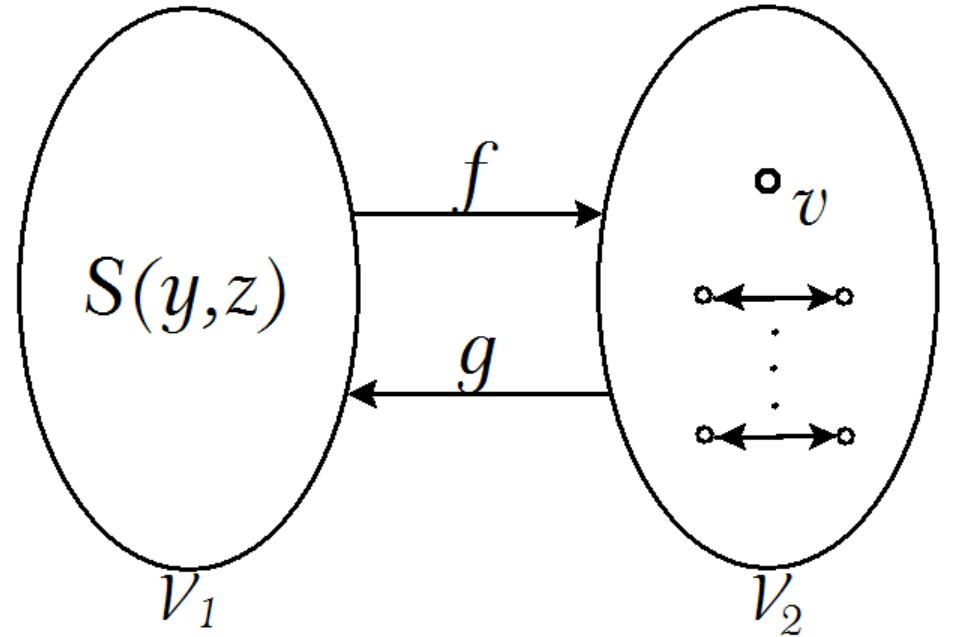} \\
		$D_1$\hspace{4cm}$D_2$\hspace{4cm}$D_3$\\
		\vspace{0.3cm}
		\includegraphics[width=1.6in]{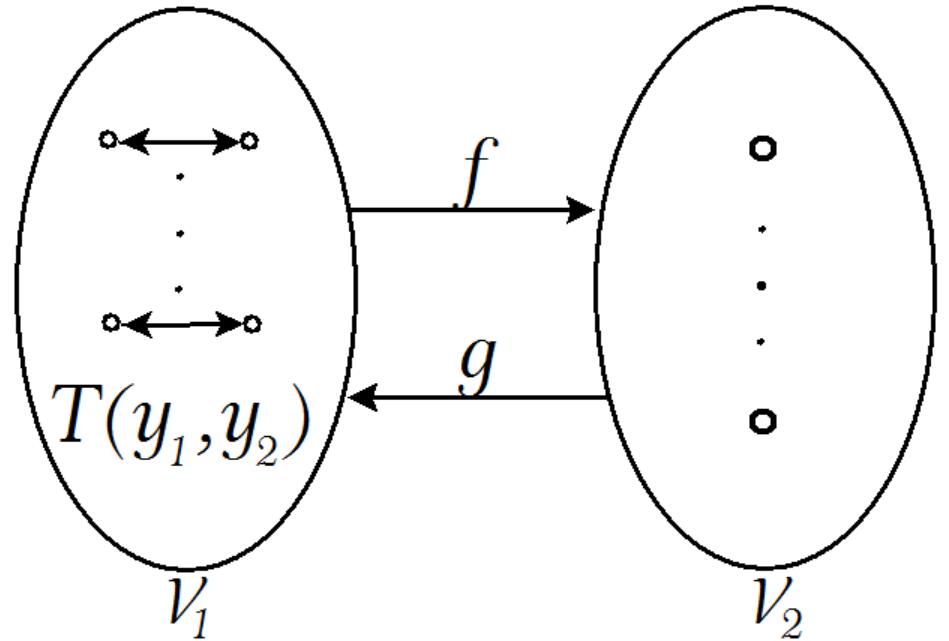}\hspace{0.1cm}
		\includegraphics[width=1.6in]{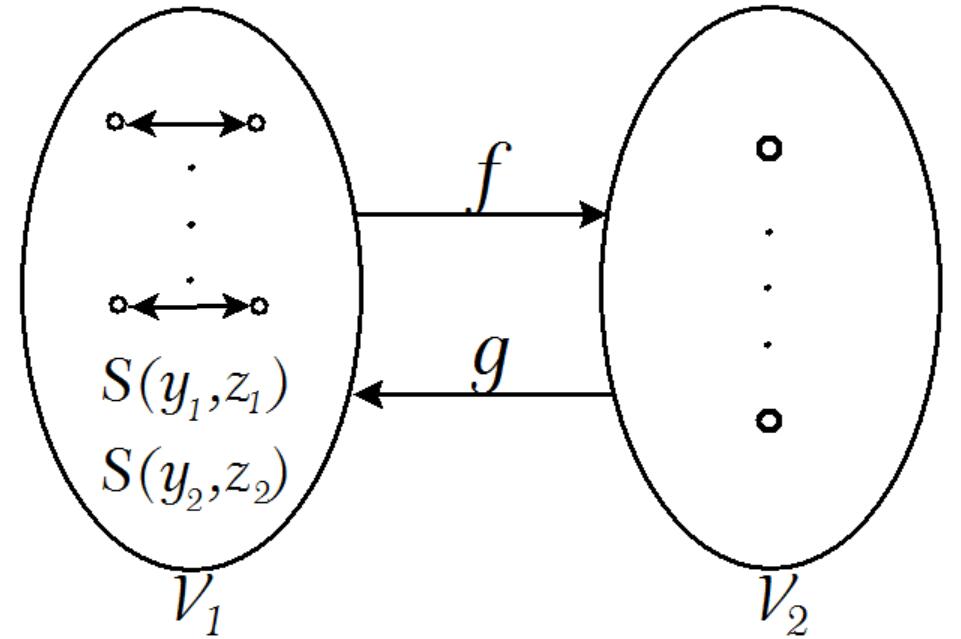}\hspace{0.1cm}
		\includegraphics[width=1.6in]{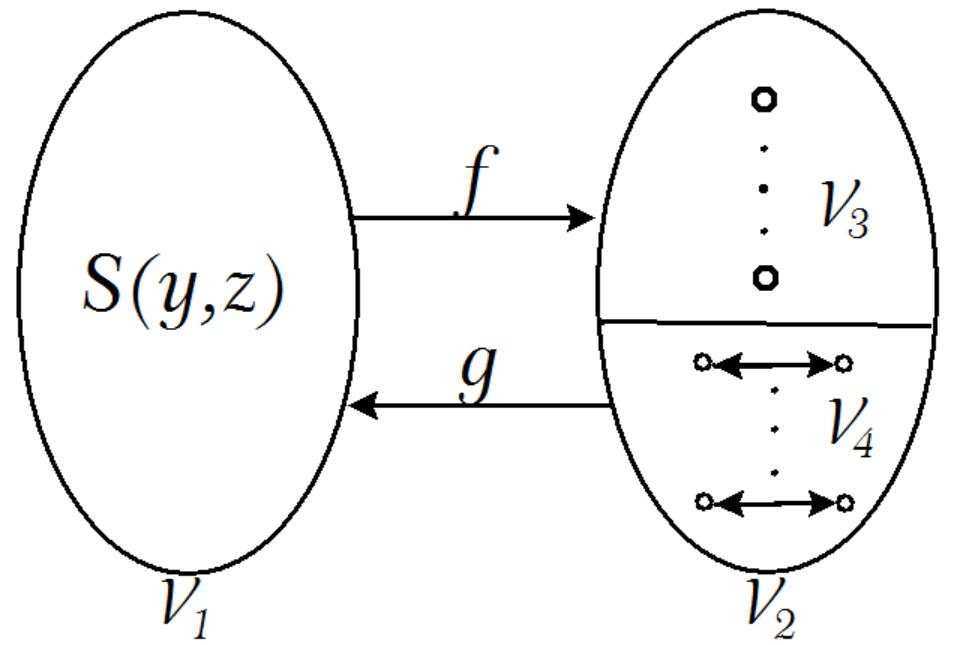} \\
		$D_4$\hspace{4cm}$D_5$\hspace{4cm}$D_6$\\
	\vspace{0.3cm}	\includegraphics[width=1.6in]{D6.jpg}\hspace{0.1cm}
	\includegraphics[width=1.6in]{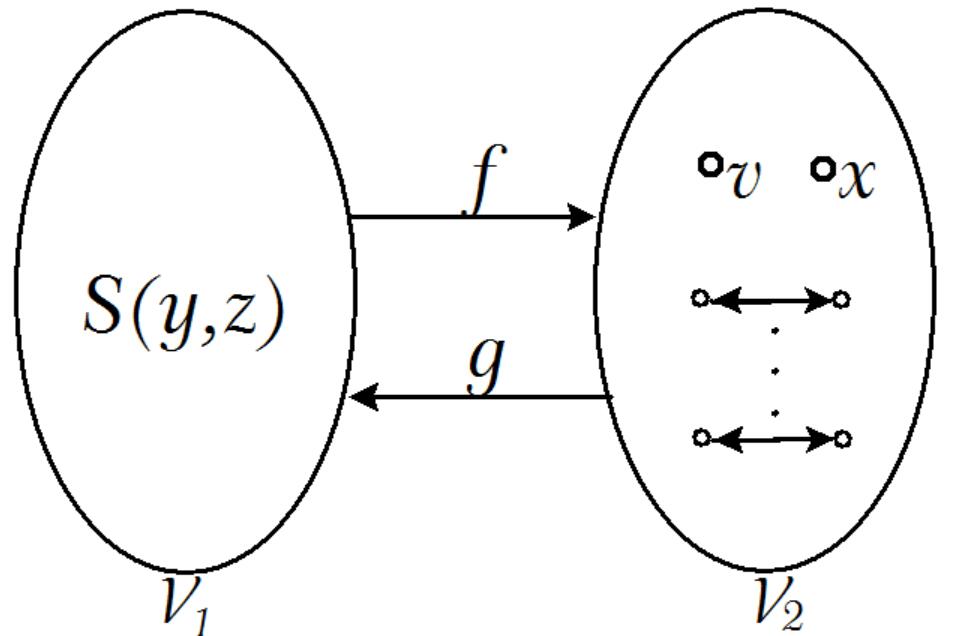}\hspace{0.1cm}
	\includegraphics[width=1.6in]{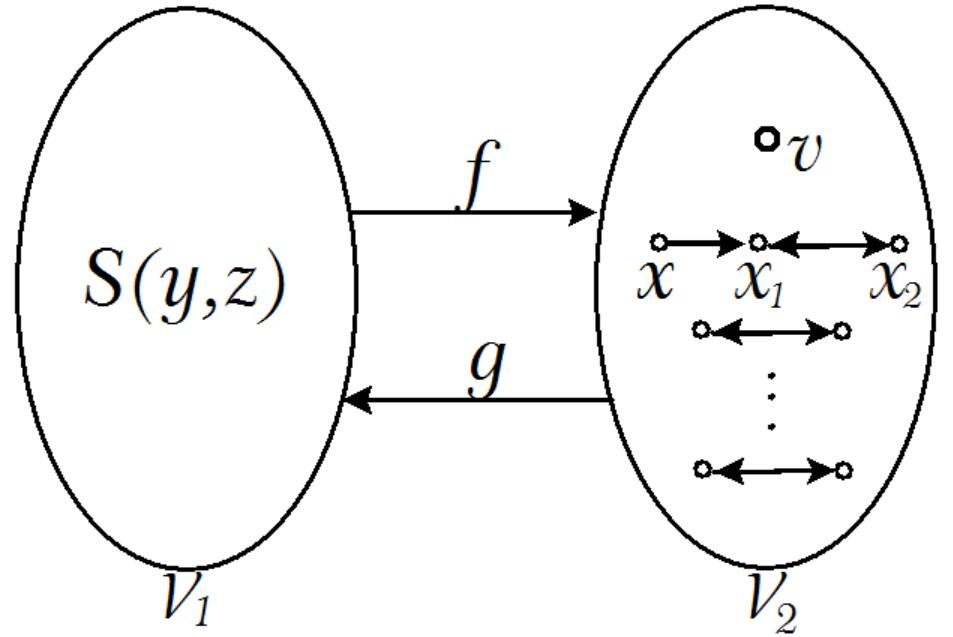} \\
	$D_7$\hspace{4cm}$D_8$\hspace{4cm}$D_9$
\end{figure}

For the order $n$ of these digraphs, it is odd for $D_1$,  $D_2$, and even for the others. Moreover, $n/2$ is even for $D_3$ and odd for $D_4,\ldots D_9$.

	In $D_1$,
	$D(\V_1)=S(y,z)$ or it is the disjoint union of $S(y,z)$ and some 2-cycles;
	$D(\V_2)$ is an empty digraph;
	$f$ means $$\V_1\backslash \{y\}~ {\rm matches} ~\V_2;$$
	$g$ means $$u\rightarrow \V_1{\rm ~for~ all~} u\in \V_2.$$

	In $D_2$, $D(\V_1)=S(y,z)$; $\V_2$ is partitioned as $\V_2=\V_3\cup \V_4$ such that $D(\V_3)$ is empty and $D(\V_4)$ is the disjoint union of 2-cycles;  $f$ means $$\V_1\backslash \{y\}~ {\rm matches} ~\V_2{\rm~ with~}  z\rightarrow w, w\in \V_3;$$
	$g$ means $$u\rightarrow \V_1{\rm ~for~ all~}u\in V_3 {\rm~and~} u\rightarrow \V_1\backslash \{z\}{\rm ~for~ all~}  u\in \V_4.$$

	In $D_3$, $D(\V_1)=S(y,z)$;  $D(\V_2)$ is the disjoint union of $(n/4-1)$ 2-cycles and an isolated vertex $v$; $f$ means
	$$z\rightarrow v~{\rm and} ~\V_1\backslash \{y,z\}~ {\rm matches} ~\V_2;$$ $g$ means
	$$v\rightarrow \V_1~{\rm and} ~u\rightarrow \V_1\backslash \{ z\}{\rm ~for~ all~} u\in\V_2\setminus\{v\}.$$

	In $D_4$,  $D(\V_1)=T(y_1,y_2)$ or it  is  the disjoint union of $T(y_1,y_2)$ and 2-cycles;
	$D(\V_2)$ is empty;
	$f$ means $$\V_1\backslash \{y_1,y_2\}{\rm ~matches~} \V_2;$$
	$g$ means $$u\rightarrow \V_1{\rm ~for~ all~} u\in \V_2.$$

	In $D_5$,  $D(\V_1)$ is the disjoint union of $S(y_1,z_1)$, $S(y_2,z_2)$ and  some 2-cycles, where the 2-cycles may vanish;  $D(\V_2)$ is empty;
	$f$ means
	$$\V_1\backslash \{y_1,y_2\}{\rm ~matches~} \V_2;$$  $g$ means  $$u\rightarrow \V_1{\rm ~for~ all~}u\in \V_2.$$

	In $D_6$,
	$D(\V_1)=S(y,z)$; $\V_2$ is partitioned as $\V_2=\V_3\cup \V_4$ such that $D(\V_3)$ is empty and $D(\V_4)$ is the disjoint union of 2-cycles, which may vanish;  $f$ means
	$$\V_1\backslash \{y,z\}{\rm ~matches~}\V_2;$$
	$g$ means $$u\rightarrow \V_1\backslash \{z\}{\rm ~for~ all~}u\in \V_4{\rm ~and~} u\rightarrow \V_1{\rm ~for~ all~}u\in \V_3.$$

	In $D_7$, $D(\V_1)$ and $D(\V_2)$ have the same structures as in $D_6$;  $f$ means
	$$\V_1\setminus\{x,y\} {\rm~ matches~} \V_2{\rm ~with~}z\rightarrow w, w\in \V_3,$$
	where $x$ is an arbitrary vertex in $\V_1\setminus\{y,z\}$;
	$g$ means $$u\rightarrow \V_1\backslash \{z\}{\rm ~for~ all~}u\in \V_4{\rm ~and~} u\rightarrow \V_1{\rm ~for~ all~}u\in \V_3.$$

	In $D_8$, $D(\V_1)=S(y,z)$; $D(\V_2)$ is the disjoint union of 2-cycles and two isolated vertices $v$ and $x$; $f$ means
	$$z\rightarrow v {\rm~and~} \V_1\backslash \{y,z\}{\rm~ matches~}\V_2;$$
	$g$ means $$v\rightarrow \V_1,~~u\rightarrow \V_1\backslash \{z\}{\rm ~for~ all~}u\in\V_2\setminus\{v,x\},$$
	and $$x\rightarrow \V_1\backslash \{z\}{\rm~or~} x\rightarrow \V_1\backslash \{w\} {\rm ~with~} w\in \V_1\setminus\{y,z\}{\rm ~ such~ that~}w\rightarrow v. $$

	In $D_9$, $D(\V_1)=S(y,z)$;
	$D(\V_2)$  is the disjoint union of some 2-cycles, a isolated vertex $v$ and the subgraph $D(\{x,x_1,x_2\})$ as in the diagram;
	$f$ means $$z\rightarrow v{\rm ~and~} \V_1\backslash \{y,z\} {\rm~ matches~}\V_2;$$
	$g$ means     $$v\rightarrow \V_1,~   u\rightarrow \V_1\backslash \{z\}{\rm ~for~ all~}u\in\V_2\setminus \{v,x\},$$and $$x\rightarrow \V_1\setminus\{z,x_2^*\}$$ with $x_2^*$ being the predecessor of $x_2$ in $\mathcal {V}_1$.
	
In addition, we need another digraph $D_{10}$, which shares the same structure  with $D_1$. In $D_{10}$, $D(\V_1)=S(y,z)$ or it is the disjoint union of $S(y,z)$ and some 2-cycles;
$D(\V_2)$ is an empty digraph; $f$ means
$$\V_1\backslash \{y,w\}{\rm ~matches~} \V_2$$
with $w$ being an arbitrary vertex in $\V_1\setminus\{y,z\}$;  $g$ means  $$u\rightarrow \V_1{\rm ~for~ all~}u\in \V_2.$$
	
	Note that each of the above diagrams represent  a class of digraphs. For convenience, we  also use $D_1, D_2,\ldots, D_{10}$ to indicate a specific digraph with the same structure as in the diagrams if it makes no confusion.
	
	Giving a digraph $D$, we denote by $D'$ the reverse of $D$, which is obtained by reversing the directions of all arcs of $D$. Given two digraphs $D$ and $H$, we say that   $D $ is an isomorphism of $ H$ if there exists a bijection $f$: $\V(D)\rightarrow\V(H)$ such that $(u,v)\in \A(D)$ if and only if $(f(u),f(v))\in \A(H)$.\\
	
	Now we state our main result as follows.

	\begin{theorem}\label{th1}
		Let  $n\ge 13$ be an integer. Then
		\begin{equation}\label{eq0}
		ex(n)=
		\begin{cases}
		\frac{n^2+4n-1}{4},&if~ n ~is ~odd;\\
		\frac{n^2+4n}{4},&if~\frac{n}{2}~is~even;\\
		\frac{n^2+4n-4}{4},&if~\frac{n}{2}~is~odd.
		\end{cases}
		\end{equation}
		Moreover, $D\in EX(n)$ if and only if
		\begin{itemize}
			\item[(1)] $n$ is odd, and $D$ or $D'$ is an isomorphism of $D_1$ or $D_2$;
			\item[(2)] $n/2$ is even, and $D$ or $D'$ is an isomorphism of $D_3$;
			\item[(3)] $n/2$ is odd, and $D$ or $D'$ is an isomorphism of $D_i$ with $i\in\{4,5,\ldots,10\}$.
		\end{itemize}
	\end{theorem}

	{\it Remark.}  For digraphs with order less than 13, (\ref{eq0}) may not be true. For example, let $D$ be the digraph with vertex set $\{1,2,3,4,5\}$ and arc set $$\{1\leftrightarrow 2,1\leftrightarrow 3,2\leftrightarrow 3,1\leftrightarrow 4,1\leftrightarrow 5,4\leftrightarrow 5\},$$ whose diagram is the following.
	\begin{figure}[H]
		\centering
		\includegraphics[width=2in]{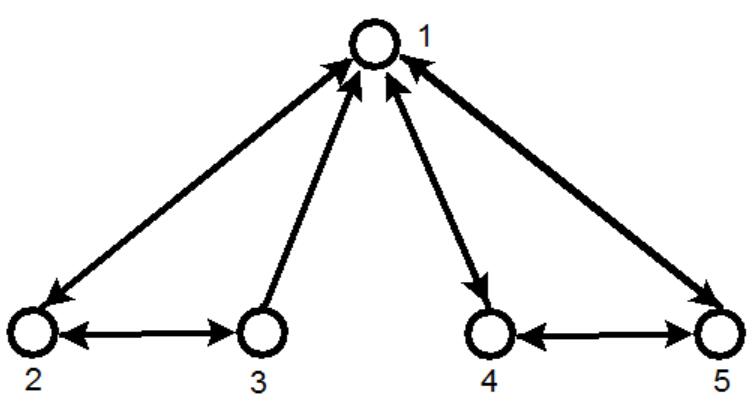}\hspace{0.8cm}
	\end{figure}
	
	\noindent It is easy to see that $D$ is $P_{2,2}$-free and it has 12 arcs, while $({n^2+4n-1})/{4}=11$ when $n=5$.

	\section{Proofs}

	In this section we give the proof of Theorem \ref{th1}. We need the following lemmas.
	
	\begin{lemma}\label{le1}
		Let $n\ge 13$ be a positive integer. Then $D_i$ is $P_{2,2}$-free  for $i=1,2,\ldots,10$ and
		\begin{equation}\label{eqh3.1}
		ex(n)\ge
		\begin{cases}
		\frac{n^2+4n-1}{4},&~if~ n ~is ~odd;\\
		\frac{n^2+4n}{4},&~if~ \frac{n}{2}~is~even;\\
		\frac{n^2+4n-4}{4},&\ if~\frac{n}{2}~is~odd.
		\end{cases}
		\end{equation}
	\end{lemma}
	
	\begin{proof}
		Suppose $D\in \{D_1,D_2,\ldots,D_{10}\}$ contains $P_{2,2}$ as its subgraph. Then we have
		$$u_1\rightarrow u_2\rightarrow u_4{\rm~ and ~}u_1\rightarrow u_3\rightarrow u_4$$ with $u_1\ne u_4$ and $u_2\ne u_3$.

		Since each vertex of $D$ has at most one successor in $\V_2$, we have
		\begin{equation*}\label{eqhh0}
		\{u_2,u _3\}\nsubseteq\V_2.
		\end{equation*} If $D\in\{D_1,D_2,D_4,D_5,D_6,D_7, D_{10}\}$,  then each every vertex has at most one predecessor in $\V_1$. Hence, \begin{equation}\label{eqhh1}
		\{u_2,u _3\}\nsubseteq\V_1.
		\end{equation} If $D\in\{D_3,D_8,D_9\}$, then $v$ has two predecessors, say, $z,w \in\V_1$. Moreover, each vertex in $\V\setminus\{v\}$ has at most one predecessor in $\V_1$. Since $z$ and $w$ cannot be both the successors of any   vertex in  $\V\setminus\{v\}$,   we have    (\ref{eqhh1}).
		
		Without loss of generality, we assume $u_2\in \V_1$ and $u_3\in \V_2$. Now we distinguish four cases.
		
		{\it Case 1.} $u_1,u_4\in \V_1$.  Then each 2-path in $D(\V_1)$ originates at a vertex who has no successor in $\V_2$, which contradicts   $u_1\rightarrow u_3$ with $u_3\in \V_2$.

		{\it Case 2.}  $u_1\in \V_1$ and $u_4\in \V_2$. If $D\in\{D_1, D_4, D_5, D_{10}\}$, then  $D(\V_2)$ contains no arc, which contradicts   $u_3\rightarrow u_4$ with $u_3,u_4\in \V_2$. If $D\in\{D_2, D_3, D_6,D_7,D_8, D_9\}$, then only two vertices in $\V_1$ have successors in $\V_1$, which are $y$ and $z$. We have $y\rightarrow z$ and $z\rightarrow \V_1\backslash \{z\}$. Since $y$ has no successor in $\V_2$, we have $u_1=z$. If  $D=D_6$, $u_1$ has no successor in $\V_2$. In other cases, $u_1$'s successor in $\V_2$ has no successor in $\V_2$.

		{\it Case 3.}  $u_1\in \V_2$ and $u_4\in \V_1$. If $D\in\{D_1,D_4,D_5, D_{10}\} $, then $D(\V_2)$ contains no arcs, which contradicts   $u_1\rightarrow u_3$. For  other cases, if $u_1$ has a successor $u_2$ in $\V_2$, then we have $N^+_{\V_1}(u_1)\nrightarrow N^+_{\V_1}(u_2)$, which contradicts $u_1\rightarrow u_2\rightarrow u_4$ and $u_1\rightarrow u_3\rightarrow u_4$.

		{\it Case 4.}  $u_1,u_4\in \V_2$. If $D\in\{D_1,D_4,D_5, D_{10}\} $, then $D(\V_2)$ contains no arc, which  contradicts $u_1\rightarrow u_3$. If $D\in\{D_2, D_3, D_6, D_7,D_8\}$,  then $D(\V_2)$ is the disjoint union of 2-cycles and isolated vertices and it contains no 2-path, a contradiction. If  $D=D_9$, the only 2-path in $D(\V_2)$ is $x\rightarrow x_1\rightarrow x_2$. Since $x\nrightarrow x_2^*$ and $x_2^*$ is the unique predecessor of $x_2$ in $\V_1$, we also get a contradiction.
		
		In all the above cases we get contradictions. Hence $ D_1,D_2,\ldots,D_{10}$ are  $P_{2,2}$-free.  By directed computation, we obtain
		$$e(D_1)=\frac{n^2+4n-1}{4}, e(D_3)=\frac{n^2+4n}{4} {\rm~ and~}  e(D_4)=\frac{n^2+4n-4}{4}.$$ Therefore, we have (\ref{eqh3.1}).
		
	\end{proof}
	
	The following lemma is obvious.
	\begin{lemma}\label{le2}
			Let $D=(\V,\A)$ be a $P_{2,2}$-free digraph. Then
		\begin{itemize}
			\item[(i)]   two distinct successors of a vertex $v\in \V$ share no common successor in $\V\setminus\{v\}$;
			\item[(ii)]given any $v\in \V$,  $e(N^+(v),u)\le 1$ for all $u\in \V\backslash \{v\}$.
		\end{itemize}
	\end{lemma}
	
			The {\it outdegree}  and {\it indegree} of a vertex $u$, denoted by $d^+(u)$  and $d^-(u)$, is the number of arcs with tails and heads $u$, respectively.
	We use the letter $k$ to denote the maximum outdegree of $D$, i.e.,
	$$k=\max_{u\in\V}d^+(u).$$
	
	Given a vertex $u\in \V$, we  always use $\V_1(u)$ and $\V_2(u)$ to denote $N^+(u)$ and $\V\backslash \V_1(u)$, respectively.
	We  also denote by
	$\tau(u)$ the number of vertices in $D$ which are both successors and  predecessors of $u$, i.e., $$\tau(u)=|N^+(u)\cap N^-(u)|=e(N^+(u),u)=e(\V_1(u),u).$$  It is obvious that $$\tau(u)\le k {\rm ~for~all~} u\in \V.$$
	The index $u$ in $\V_1(u), \V_2(u)$ and $\tau(u)$ will be omitted if no confusion arises.
	
	\begin{lemma}\label{le7}
		Let $D=(\V,\A)$ be a digraph with a vertex $u$ such that  $d^+(u)\ge 2$. Suppose $\{v_1,v_2\}\subseteq N^+(u)$ and $S\subseteq \V$.  If $e(v_1,S)+e(v_2,S)\ge |S|+2$, then $D$ is not $P_{2,2}$-free.
	\end{lemma}
	
	\begin{proof}
		The inequality guarantees that $v_1$ and $v_2$ share at least two common successors. Applying Lemma \ref{le2},  $D$ is not $P_{2,2}$-free.
		
	\end{proof}

	\begin{lemma}\label{le9}
		Let $D=(\V,\A)$ be  a $P_{2,2}$-free  digraph and $v\in \V$.  Then  each  $u\in \V_2(v)\backslash \{v\}$ shares at most $d^+(v)-\tau(v)+1$ common successors with $v$.
	\end{lemma}
	
	\begin{proof} Assume there exists a vertex $u\in V_2(v)\backslash \{v\}$ sharing $d^+(v)-\tau(v)+2$ common successors with $v$. By the definition of $\tau(v)$, there are at least two  successors $v_1,v_2$ of  $u$ belonging to $N^+(u)\cap N^-(v)$. So there are two paths
		$$u\rightarrow v_1\rightarrow v {\rm~ and~ }u\rightarrow v_2\rightarrow v.$$ Hence, $D$ is not $P_{2,2}$-free, a contradiction.
		
	\end{proof}

	Given $v\in \V$, let $\alpha(v)=\max\limits_{u\in \V}e(u,\V_2(v))$. For convenience, we simply write $\alpha$ if $v$ is clear. We   have the following upper bound on $\alpha$.
	
	\begin{lemma}\label{le12}
		
		Let $D=(\V,\A)\in EX(n)$ with $n\ge 13$, and let $v\in \V$  such that $d^+(v)=k$. Then
		 $$\alpha(v)\le 1{~\rm and~}\tau(v)\le 2.$$
		
	\end{lemma}
	
	\begin{proof}
		
		Denote by
		$$\tau=\tau(v), \alpha_1=\max\limits_{u\in \V_1}e(u,\V_2), \alpha_2=\max\limits_{u\in \V_2}e(u,\V_2), {\rm ~and ~}\beta=\max\limits_{u\in \V_2\backslash \{v\}}d^+(u). $$
		Then  $\alpha=\max\{\alpha_1,\alpha_2\}$ and
		\begin{equation}\label{equ:3.1}
		e(D)= e(\V_2,\V)+e(\V_1,\V)=\sum\limits_{u\in \V_2}d^+(u)+\sum\limits_{u\in \V}e(\V_1,u).
		\end{equation}
		Moreover,
		\begin{equation}\label{equ:3.3}
		\tau\le k-\beta+\alpha_2+1.
		\end{equation}
		In fact, if $\beta\le \alpha_2+1$, (\ref{equ:3.3}) holds trivially. If $\beta> \alpha_2+1$, suppose  $u\in \V_2\setminus\{v\}$ such that $d^+(u)=\beta$. Then $u$ has at least $\beta-\alpha_2$ successors in $\V_1$. Applying Lemma \ref{le9} we get $$\beta-\alpha_2\le k-\tau+1,$$ which is equivalent with (\ref{equ:3.3}).

		Firstly we prove $\alpha_2\le 1$.
		If $\alpha_2\ge 4$, there exists $u_0\in \V_2$ such that $u_0\rightarrow u_i$, where $u_i\in \V_2$ for $i=1,2,3,4$.
		Since $u_i$ ($i=1,2,3,4$) shares no common successor in  $\V\backslash \{u_0\}$ with each other, we have
		$$ \sum\limits_{i=1}^{4}d^+(u_i)\le n+3.$$
		Moreover, by Lemma \ref{le2} we have
		$$e(\V_1,u)\le 1 {\rm ~ for~ all~} u\in \V\setminus\{v\}. $$Hence
		\begin{eqnarray*}
			e(D)&=& \sum\limits_{u\in \V_2\backslash\{u_1,\dots,u_4\}}d^+(u)+\sum\limits_{i=1}^{4}d^+(u_i)+\sum\limits_{u\in \V\backslash \{v\}}e(\V_1,u)+e(\V_1,v)\\
			&\le& (n-k-4)k+n+3+n-1+\tau\\
			&\le& (n-k-4)k+n+3+n-1+k\\
			&=&-(k-\frac{n-3}{2})^2+\frac{n^2+2n+17}{4}\\
			&<& \frac{n^2+4n-4}{4}.
		\end{eqnarray*}
		It follows from (\ref{eqh3.1}) that $e(D)<ex(n)$, which contradicts $D\in EX(n)$. Thus, we have $\alpha_2\le 3$.

		Now we assert that $\beta=k$.
		Otherwise, if $\beta\le k-2$, then
		\begin{eqnarray*}
			e(D)&= &\sum\limits_{u\in \V_2\backslash\{v\}}d^+(u)+d^+(v)+\sum\limits_{u\in \V\backslash \{v\}}e(\V_1,u)+e(\V_1,v)\\
			&\le&  (n-k-1)(k-2)+k+n-1+\tau\\
			&\le&  (n-k-1)(k-2)+k+n-1+k\\
			&=& -(k-\frac{n+3}{2})^2+\frac{n^2+2n+13}{4}\\
			&<&\frac{n^2+4n-4}{4};
		\end{eqnarray*}
		if $\beta=k-1$, then by (\ref{equ:3.3})  we have $\tau\le 5$ and
		\begin{eqnarray*}
			e(D)&=& \sum\limits_{u\in \V_2\backslash\{v\}}d^+(u)+d^+(v)+\sum\limits_{u\in \V\backslash \{v\}}e(\V_1,u)+e(\V_1,v)\\
			&\le& (n-k-1)(k-1)+k+n-1+5\\
			&=&-(k-\frac{n+1}{2})^2+\frac{n^2+2n+21}{4}\\
			&<& \frac{n^2+4n-4}{4}.
		\end{eqnarray*}
		In both cases we get $e(D)<ex(n)$, which contradicts $D\in EX(n)$.    Hence, $\beta=k$.

		Now suppose $\alpha_2=2$ or 3. Then by (\ref{equ:3.3}) we have $\tau\le 4$.
		Let $w\in \V_2$ such that  $$N^+_{\V_2}(w)=\{u_1,\ldots,u_{\alpha_2}\}.$$
		If $v\notin N^+_{\V_2}(w)$, then Lemma \ref{le2} ensures that no pair of  vertices in $N^+_{\V_2}(w)$ shares a common successor in $\V_1\backslash N^-_{\V_1}(v)$ and each $u_i$ has at most one successor in $N^-_{\V_1}(v)$ for $1\le i\le \alpha_2$, which implies $$\sum\limits_{i=1}^{\alpha_2}e(u_i,\V_1)=\sum\limits_{i=1}^{\alpha_2}e(u_i,\V_1\backslash N^-_{\V_1}(v))+\sum\limits_{i=1}^{\alpha_2}e(u_i,N^-_{\V_1}(v))\le k-\tau+\alpha_2.$$ Therefore,
		\begin{eqnarray}\label{equ:3.2}
		\sum\limits_{i=1}^{\alpha_2}d^+(u_i)=\sum\limits_{i=1}^{\alpha_2}e(u_i,\V_1)+\sum\limits_{i=1}^{\alpha_2}e(u_i,\V_2)\le k-\tau+\alpha_2+(\alpha_2)^2.
		\end{eqnarray}
		By (\ref{equ:3.1}), we obtain
		\begin{eqnarray*}
			e(D)&=&\sum\limits_{u\in \V_2\backslash \{u_1,\ldots,u_{\alpha_2}\}}d^+(u)+\sum\limits_{u\in \{u_1,\dots,u_{\alpha_2}\}}d^+(u)+\sum\limits_{u\in \V}e(\V_1,u)\\
			&\le& k(n-k-\alpha_2)+k-\tau+\alpha_2+(\alpha_2)^2+\tau+n-1\equiv f_1.
		\end{eqnarray*}
			If $v\in N^+_{\V_2}(w)$, say, $v=u_1$. By Lemma \ref{le2},   $\sum\limits_{u\in \{u_2,\dots,u_{\alpha_2}\}}e(u,\V_1)=0$. Note $v$ has no successor in $\V_2$. We have
		\begin{eqnarray}
		\sum\limits_{i=1}^{\alpha_2}d^+(u_i)=\sum\limits_{i=1}^{\alpha_2}e(u_i,\V_1)+\sum\limits_{i=1}^{\alpha_2}e(u_i,\V_2)\le k+(\alpha_2-1)\alpha_2.
		\end{eqnarray}
		By  (\ref{equ:3.1}), we obtain
		\begin{eqnarray*}
			e(D)&=&\sum\limits_{u\in \V_2\backslash \{u_1,\ldots,u_{\alpha_2}\}}d^+(u)+\sum\limits_{u\in \{u_1,\dots,u_{\alpha_2}\}}d^+(u)+\sum\limits_{u\in \V}e(\V_1,u)\\
			&\le& k(n-k-\alpha_2)+k-\alpha_2+(\alpha_2)^2+n-1+\tau\equiv f_2.
		\end{eqnarray*}
	We can verify that
	$$f_i<ex(n) ~{\rm for~}  i=1,2 {\rm ~and~ }\alpha_2=2,3,$$  which  contradicts $D\in EX(n)$.
		Hence,   $$\alpha_2\le 1~{\rm and}~\tau\le 2.$$

		Next we show that $\alpha_1\le 1$. Otherwise, suppose there exists $u\in \V_1$ such that $\V_2\cap N^+(u) $ has two distinct vertices $u_1,u_2$. By Lemma \ref{le2}, we have $$\sum\limits_{i=1}^{2}e(u_i,\V_1)\le k+1.$$ Since $\alpha_2\le 1$, we obtain $$e(u_1,\V_2)+e(u_2,\V_2)\le 2$$ and $$d^+(u_1)+d^+(u_2)\le k+3.$$ Again, from (\ref{equ:3.1})  we have
		\begin{eqnarray*}
			e(D)&=& \sum\limits_{u\in \V_2\backslash\{u_1,u_2\}}d^+(u)+\sum\limits_{i=1}^{2}d^+(u_i)+\sum\limits_{u\in\V\backslash \{v\}}e(\V_1,u)+e(\V_1,v)\\
			&\le&(n-k-2)k+k+3+n-1+\tau\\
			&<&ex(n),
		\end{eqnarray*}
		a contradiction.
		
		Therefore, $\alpha=\max\{\alpha_1,\alpha_2\}\le 1$. This completes the proof.
	\end{proof}

	\begin{lemma}\label{le13}
		Let $D=(\V,\A)\in EX(n)$ with $n\ge 13$. Then
	\begin{equation}\label{eqhh}
		\frac{n}{2}\le k\le \frac{n}{2}+2.
	\end{equation}
	\end{lemma}
	
	\begin{proof}
		
		Let  $v$ be a vertex of $D$ such that $d^+(v)=k$.  Then $$e(\V_2,\V)=\sum\limits_{u\in \V_2}e(u,\V)\le (n-k)k.$$  Applying Lemma \ref{le2} and Lemma \ref{le12} we have
		$$e(\V_1)=\sum\limits_{u\in \V_1}e(\V_1,u)\le k$$
		and
		$$e(u,\V_2)\le 1~ {\rm for ~all} ~ u\in \V_1,$$ which implies $e(\V_1,\V_2)\le k$. It follows that
		\begin{eqnarray*}
			e(D)&=&e(\V_2,\V)+e(\V_1,\V_2)+e(\V_1)\\
			&<&(n-k)k+2k,
		\end{eqnarray*}
		which is less than  $ex(n)$ when $k<n/2$ or $k>n/2+2$. Hence, we get (\ref{eqhh}).

	\end{proof}

	Let  $D$ be a digraph with maximum outdegree $k$, let $v$ be a vertex in $D$ such that $d^+(v)=k$, and let $u\in \V_2(v)$. If $d^+_{\V_1(v)}(u)=k-1$, then we denoted by $u'$ the unique vertex of $\V_1(v)\backslash N^+_{\V_1(v )}(u)$.
	
	\begin{lemma}\label{le10}
		
		Let $D\in EX(n)$ and $v$ be a vertex such that $d^+(v)=k$. If $u_1,u_2\in \V_2(v)$ and $u_1\rightarrow u_2$, then $$N^+_{\V_1(v)}(u_1)\nrightarrow N^+_{\V_1(v)}(u_2).$$ Moreover, if $\V_1(v)\rightarrow \V_1(v)$ and $d^+(u_1)=d^+(u_2)=k$, then $$N^+_{\V_1(v)}(u_1)=N^+_{\V_1(v)}(u_2)=\V_1(v)\backslash \{u_1'\}{\rm~ and~} u_1'\rightarrow \V_1(v)\backslash \{u_1'\}.$$
		
	\end{lemma}

	\begin{proof}
		
		Suppose there exist $u_3\in N^+_{\V_1(v)}(u_1)$ and $u_4\in N^+_{\V_1(v)}(u_2)$ such that $u_3\rightarrow u_4$. Then we have $$u_1\rightarrow u_2\rightarrow u_4{\rm~ and~} u_1\rightarrow u_3\rightarrow u_4,$$a contradiction with $D\in EX(n)$. Hence, $N^+_{\V_1(v)}(u_1)\nrightarrow N^+_{\V_1(v)}(u_2)$.

		For  the second part,
		since $d^+(u_1)=d^+(u_2)=k$, by  Lemma \ref{le12} we have $$d^+_{\V_2}(u_i)\le 1{\rm~  and ~}d^+_{\V_1}(u_i)\ge k-1{\rm ~for~} i=1,2.$$
		Now  $\V_1(v)\rightarrow \V_1(v)$ and $N^+_{\V_1(v)}(u_1)\nrightarrow N^+_{\V_1(v)}(u_2)$ imply $$d^+_{\V_1}(u_1)=d^+_{\V_1}(u_2)=k-1~~{\rm and}~~u_1'\rightarrow N^+_{\V_1(v)}(u_2) .$$
		Since $D$ is loopless, we have $u_1'\notin N^+_{\V_1(v)}(u_2)$. Hence $u_1'=u_2'$    and $$N^+_{\V_1(v)}(u_1)=N^+_{\V_1(v)}(u_2)=\V_1(v)\backslash \{u_1'\}.$$
	\end{proof}

	Now we are ready to present the proof of Theorem \ref{th1}.\\
	
	{\bf Proof of Theorem \ref{th1}.}
	Let $D=(\V,\A)\in EX(n)$.
	Note that a digraph is in $EX(n)$ if and only if its reverse   is also in $EX(n)$. Without loss of generality, we may assume the maximum outdegree of $D$ is larger than or equal to its maximum indegree.

	Let  $v\in \V$ such that $d^+(v)=k$.
	Denote by $$\V_3=\{u\in \V_2|N^+(u)=\V_1\}{\rm ~and~} \V_4=\V_2\backslash \V_3.$$  By Lemma \ref{le2} and Lemma \ref{le12}, we have
	$$e(\V_1,\V\backslash \{v\})\le n-1{\rm~ and ~}\tau\le 2.$$  It follows that
	\begin{eqnarray}
	\nonumber e(D)&=&e(\V_2,\V)+e(\V_1,\V\backslash \{v\})+e(\V_1,v)\\
	\nonumber &\le& k(n-k)+n-1+\tau\\
	&=&-(k-\frac{n}{2})^2+\frac{n^2+4n}{4}+\tau-1. \label{eq4.1}
	\end{eqnarray}
	We distinguish tow cases according to the parity of $n$.
	
	(1) $n$ is odd.  Then by (\ref{eq4.1}), we have
	\begin{equation}\label{eqhhh3.9}
e(D)\le \frac{ n^2+4n+3}{4}.
	\end{equation}
	If  equality in (\ref{eqhhh3.9}) holds, then the equalities in (\ref{eq4.1}) imply $\tau=2$ and
	\begin{equation}\label{h4.2}
	(k-\frac{n}{2})^2= \frac{1}{4}, e(\V_1,\V\backslash \{v\})=n-1, {\rm ~and ~}  e(\V_2,\V)=k(n-k).
	\end{equation}
	Combining (\ref{h4.2}) with  Lemma \ref{le13} and Lemma \ref{le2}, we have $k=(n+1)/{2}$ and
	\begin{equation}\label{eq4.3}
	e(\V_1,u)=1{\rm~ for~ all~} u\in \V\backslash \{v\},
	\end{equation}
	which implies $$e(\V_1,\V_2\backslash \{v\})=n-k-1.$$    It follows that $$e(\V_1,\V_2)=\tau+e(\V_1,\V_2\backslash \{v\})=\frac{n+1}{2}.$$ Since $|\V_1|=({n+1})/{2}$, applying Lemma \ref{le12} we obtain that each $u\in \V_1$ has exactly one successor in $\V_2$.
	
	Note that $e(\V_2,\V)=k(n-k)$ implies
	\begin{equation}\label{h4.4}
	d^+(u)=k~{\rm for~all~}   u\in \V_2.
	\end{equation} Given any vertex $u_1\in \V_2\backslash \{v\}$, by Lemma \ref{le9},  there  exist a vertex  $u_2\in \V_2$ such that  $u_1\rightarrow u_2$. Since $\V_1\rightarrow \V_1$, by Lemma \ref{le10} we obtain
	$$N^+_{\V_1}(u_2)=N^+_{\V_1}(u_1){\rm ~and ~}u_1'\rightarrow \V_1\backslash \{u_1'\}. $$Recall that $u_1'$ has a predecessor $u_3\in \V_1$, which possesses a successor $u_4\in \V_2$, i.e., $u_3\rightarrow u_1'$ and $u_3\rightarrow u_4$. Since $$e(u_1',\V_1)=k-1{\rm~ and~ }e(u_4,\V_1)\ge k-\alpha\ge k-1,$$applying Lemma \ref{le7} we have $D\notin EX(n)$, a contradiction. Therefore, $$e(D)\le \frac{n^2+4n-1}{4}.$$

	Now by  (\ref{eqh3.1}) we obtain
\begin{equation}\label{eqhh121}
	ex(n)=e(D)=\frac{n^2+4n-1}{4}.
\end{equation}
	Moreover,   (\ref{eq4.1}) leads to $$k=(n+1)/{2} {\rm~and~} \tau\ge 1.$$ By Lemma \ref{le12},  we have
	$$\tau=1{\rm~ or~} \tau=2.$$\par
	
	We need the following claim.\\
	
	{\it {\bf Claim 1. } $D$ contains a vertex $z$ such that $$d^+(z)=k {\rm ~and ~}\tau(z)=1.$$}
	
	{\it Proof of Claim 1.} If $\tau=1$, then $v$ is the vertex we need.
	Now assume  $\tau=2$ and $N^-(v)\cap \V_1=\{v_1,v_2\}$.
	Combining (\ref{eq4.1}) and (\ref{eqhh121}), we have
	either
	\begin{equation}\label{eq4.7}
	e(\V_2,\V)= (n-k)k-1,  ~~e(\V_1,\V\backslash \{v\})= n-1
	\end{equation}
	or
	\begin{equation}\label{eq4.8}
	e(\V_2,\V)= (n-k)k, ~~ e(\V_1,\V\backslash \{v\})= n-2.
	\end{equation}

	If (\ref{eq4.7}) holds, then there is exactly one vertex $x\in \V_2$ with outdegree $k-1$ and  all vertices in $\V_2\setminus\{x\}$ have outdegree $k$.  By Lemma \ref{le2}, we have $\V_1\rightarrow \V$.  Given any $u\in \V_2\backslash \{v,x\}$, since $d^+(u)=k$, $\alpha\le 1$, and $v_1,v_2$ cannot be the successors of $u$ simultaneously, we see that $u$ has exactly one successor in $\V_2$, say, $u\rightarrow u_1\in \V_2$. By Lemma \ref{le10}, we obtain $\mathcal{V}_1\backslash \{u'\}\nrightarrow N^+_{\mathcal{V}_1}(u_1)$. Since $\V_1\rightarrow \V_1$, we have
	\begin{equation}\label{eqhhhh1}
	 u'\rightarrow N^+_{\V_1}(u_1)
	\end{equation}
	and there exists $t\in \V_1$ such that $t\rightarrow u'$. Note that $|\V_1|=k$ and $e(\V_1,\V_2)=k$. Lemma \ref{le12} guarantees $t$ has a successor $u_2\in \V_2$. By (\ref{eqhhhh1}) and
	$$e(u_2,\V_1)\ge d^+(u_2)-\alpha\ge k-2,$$ applying Lemma \ref{le7} we obtain $D\notin EX(n)$, a contradiction.

	Now suppose (\ref{eq4.8}) holds.  Then all vertices in $ \V_2$ have outdegree $k$.
	By Lemma \ref{le2} and $\alpha\le 1$, the second equality in  (\ref{eq4.8}) implies there exists exactly one vertex $x$ such that
	\begin{equation}\label{eq4.9}
	e(\V_1,x)=0
	\end{equation} and
	\begin{equation}\label{eqhh4.9}
 \V_1\rightarrow \V\setminus\{x\}.
	\end{equation}
	
	Suppose $x\in \V_2$. Then $$\tau(x)\le e(\V_1,x)+e(N^+_{\V_2}(x),x)\le 1$$ as $\alpha\le 1$. Replacing the role of $v$ by $x$ in (\ref{eq4.1}), we have
	\begin{eqnarray}
	\nonumber e(D)&=&e(\V_2(x),\V)+e(\V_1(x),\V\backslash \{x\})+e(\V_1(x),x).
	\end{eqnarray}
	If $\tau(x)=0$, then $e(D)<ex(n)$, a contradiction. Hence, $\tau(x)=1$, and $x$ is the vertex we need.

	Next we assume   $x\in \V_1$. Then $\V_1\rightarrow \V_2$.  Since $\alpha\le 1$, we have $e(\V_1,\V_2)=k$ and every vertex in $\V_1$ has a successor in $\V_2$.
	Since  $\tau=2$ and all vertices in $\V_2$ have outdegree $k$, by Lemma \ref{le9}, each vertex in $\V_2\backslash \{v\}$ has a successor in $\V_2$. Let $(u_1,u_2)\in  D(\V_2)$. By Lemma \ref{le2}, we have $u_1\nrightarrow v_1$ or $u_1\nrightarrow v_2$. Without loss of generality, we assume $u_1\nrightarrow v_2$. Then $N^+_{\V_1}(u_1)=\V_1\backslash \{v_2\}$. Applying Lemma \ref{le10}, we have $$\V_1\backslash \{v_2\}\nrightarrow N^+_{\V_1}(u_2). $$Hence, by (\ref{eq4.9}) we have  \begin{equation}\label{eqh4.10}
	v_2\rightarrow N^+_{\V_1}(u_2)\backslash \{x\}.
	\end{equation}

	We assert that $v_2$ has no predecessor in $\V_1$. Otherwise, suppose $v_2$ has a predecessor $v_3\in \V_1$. Note that $v_3$ has a successor $v_4\in \V_2$ and $e(v_4,\V_1)\ge k-1$. By Lemma \ref{le7} we obtain $D\notin EX(n)$, a contradiction. Hence,  $v_2=x$.

	Next we assert $v_2\rightarrow \V_1\backslash \{v_2\}$. Otherwise there exists  $t\in \V_1\backslash \{v_2\}$ such that $v_2\nrightarrow t$. Since $N^+_{\V_1}(u_2)\ge k-1$, by (\ref{eqhh4.9}) and (\ref{eqh4.10}), one of $v_2$'s successor $w\in \V_1$ is a predecessor of $t$. Then we have $$v_2\rightarrow v\rightarrow t{\rm ~ and~} v_2\rightarrow w\rightarrow t,$$ which contradicts $D\in EX(n)$.

	Therefore, we have $N^+(v_2)=\{v\}\cup \V_1\backslash \{v_2\}$ and $d^+(v_2)=k$. Moreover,
	$$ \tau(v_2)=e(N^+(v_2),v_2)=e(\V_1\backslash \{v_2\},v_2)+e(v,v_2)=1.$$
	Thus $v_2$ is the vertex we need. This completes the proof of Claim 1.\\

	By Claim 1, without loss of generality, we may assume $\tau=1$, since otherwise we may replace the role of $v$ with $z$ so that  the new digraph is an isomorphism of $D$.

	From (\ref{eq4.1}), we have (\ref{h4.2}), which implies (\ref{h4.4}). By Lemma \ref{le2}, the second equation  in (\ref{h4.2})  indicates $\V_1\rightarrow \V$, which means
\begin{equation}\label{eqh*}
	\V_1\rightarrow \V_1 ~and~ \V_1\rightarrow \V_2.
	\end{equation}

Since $|\V_1|=|\V_2|+1$, by Lemma \ref{le12}, there exists exactly one vertex $y\in \V_1$ having no successor in $\V_2$. It follows from Lemma \ref{le2} that
	\begin{equation}\label{eqhhh319}
		\V_1\setminus \{y\}{\rm ~~matches~~} \V_2.
	\end{equation} Now $\V_1\rightarrow \V_1$ implies there exists a vertex $y_0\in \V_1$ such that $y_0\rightarrow y$.

	Now we distinguish two cases.

	{\it Case 1.1.}  $\V_3=\V_2$, i.e.,
	\begin{equation}\label{eqhhh320}
	N^+(u)=\V_1~{\rm for~ all~} u\in \V_2.
	\end{equation} Since $y_0\in \V_1$ and $\V_1\rightarrow \V_1$, $y_0$ has a predecessor  $y_1\in \V_1$.

	We assert that $y_1=y$. Otherwise, we have $y_1\rightarrow y_0\rightarrow y$ and $y_1$ has a successor $y_2\in \V_2$ such that $N^+(y_2)=\V_1$. Then $y_1\rightarrow y_2\rightarrow y$ and we have two 2-paths from $y_1$ to $y$, a contradiction with $D\in EX(n)$. Hence, $y\leftrightarrow y_0$.

	Moreover,  we have
	\begin{equation}\label{eq4.5}
	e(y,\V_1)=1,~{\rm i.e.},~ N^+_{\V_1}(y)=y_0.
	\end{equation} Otherwise, suppose there is an arc $y\rightarrow y_3$ with $y_3\in \V_1\setminus \{y_0\}$. We have $y_0\rightarrow y\rightarrow y_3$. On the other hand, $y_0$ has a successor   $y_4\in\V_2$ with $N^+(y_4)=\V_1$.  Hence we have another 2-path from $y_0$ to $y_3$, which is $y_0\rightarrow y_4\rightarrow y_3$, a contradiction.

	For any $u\in \V_1\backslash \{y,y_0\}$, we assert either $y_0\rightarrow u$ or there exists $u_1\in \V_1$ such that $u\leftrightarrow u_1$. Otherwise, there exists $u_2\in \V_1$ such that $u_2\rightarrow u$ and $u_2\ne y_0$, and there exists $u_3\in \V_1$ such that $u_3\rightarrow u_2$ and $u_3\ne u$. It follows from (\ref{eq4.5}) that $u_3\ne y$. Since $\V_1\backslash \{y\}$ matches $\V_2$, $u_3$ has a successor $u_4\in \V_2$ with $N^+(u_4)=\V_1$. We have $$u_3\rightarrow u_4\rightarrow u {\rm~and~} u_3\rightarrow u_2\rightarrow u,$$ a contradiction with $D\in EX(n)$.

	By (\ref{eqh*}) we know each vertex in $\V_1$ has exactly one predecessor in $\V_1$. Hence, $D(\V_1)=S(y,y_0)$ or $D(\V_1)$ is the disjoint union of $S(y,y_0)$ and 2-cycles. Combining this with (\ref{eqhhh319}) and (\ref{eqhhh320}), we see that $D$ is an isomorphism of $D_1$.

	{\it Case 1.2.} $\V_3\ne \V_2$, i.e., there exists $u_1\rightarrow u_2$ with $u_1,u_2\in \V_2$. Applying Lemma \ref{le10}, we have
	\begin{equation}\label{eq4.6}
	u_2\in \V_4~{\rm and }~u_1'\rightarrow \V_1\backslash \{u_1'\}.
	\end{equation}
	Given any $u\in \V_4$, we have
	\begin{equation}\label{eqhhhh2}
 N^+_{\V_1}(u)= \V_1\backslash \{u_1'\}.
	\end{equation} Otherwise,  $N^+_{\V_1}(u)\ne \V_1\backslash \{u_1'\}$ means $u'\ne u_1'$. Since $u\in \V_4$ has a successor $u_3\in \V_2$,  applying Lemma \ref{le10}  we have $u'\rightarrow \V_1\backslash \{u'\}$. It follows that $|(\V_1\setminus\{u'\})\cap (\V_1\setminus\{u_1'\})|\ge 1$, a contradiction with Lemma \ref{le2}.

	If $y=u_1'$, since $\V_1$ matches $\V_2$, $y_0$ has a successor $y_1\in \V_2$. Moreover, $\V_1\setminus \{y\}\subseteq N^+(y_1)$. Therefore, we have
	$$y_0\rightarrow y\rightarrow y_2~{\rm and~} y_0\rightarrow y_1\rightarrow y_2 ~{\rm for~all~} y_2\in \V_1\setminus \{y,y_0\},$$
	a contradiction. Thus   $y\ne u_1'$, and (\ref{eq4.6}) implies $u_1'\rightarrow y$.  By Lemma \ref{le2}, $y$ has only one predecessor in $\V_1$. Hence $u_1'=y_0$ and
	\begin{equation}\label{eqhhh323}
	D(\V_1)=S(y,y_0).
	\end{equation}

By (\ref{eqhhh319}), $u_1'$ has a successor $u_1^*\in \V_2$. We assert $u_1^*\in \V_3$. Otherwise $u_1^*\in \V_4$ has a successor $t_1\in \V_2$, which has a predecessor $t_2\in \V_1\setminus\{u_1'\}$.  Hence, we have
$$u_1'\rightarrow u_1^*\rightarrow t_1 {\rm~and~}u_1'\rightarrow t_2\rightarrow t_1,$$ a contradiction with $D\in EX(n)$.

	For any $x_1\rightarrow x_2$ in $D(\V_2)$, applying Lemma \ref{le10} we have $x_2\in \V_4$, which has a successor $x_3\in \V_4$. Then $x_3$ is not  a successor of  $u_1'$. Using (\ref{eqhhh319}) again, we have $x_1\rightarrow \V_1\backslash \{u_1'\}\rightarrow x_3$. If $x_1\ne x_3$, we have two 2-paths from $x_1$ to $x_3$. Hence, $x_1=x_3$. Since $x_1$ is arbitrarily chosen, we conclude that $D(\V_2)$ is the union of 2-cycles and isolated vertices. By Lemma \ref{le12}, these 2-cycles are disjoint. Therefore, by (\ref{eqhhh319}), (\ref{eqhhhh2}) and (\ref{eqhhh323}), $D$ is an isomorphism of $D_2$.
	\\

	(2) $n$ is even. Then by (\ref{eq4.1}) we get \begin{equation}\label{eq4.11}
	e(D)\le\frac{n^2+4n}{4}+1.
	\end{equation}

	If equality in (\ref{eq4.11}) holds, then
	$$k=\frac{n}{2}{\rm ~and~} \tau=2.$$
	Moreover, $e(\V_1,\V_1)\le k$ and  (\ref{eq4.1}) lead to $e(\V_1,\V_2)=k+1$, which implies that  there exists a vertex in $\V_1$ with at least two successors in $\V_2$, which contradicts $\alpha\le 1$.

	Now suppose $$e(D)=\frac{n^2+4n}{4}.$$
	Then (\ref{eq4.1}) leads to either
	\begin{equation*}
	k=\frac{n}{2}, ~~ \tau\in \{1,2\}
	\end{equation*}
	or
	\begin{equation}\label{eq4.12}
	k=\frac{n}{2}+1,~~ \tau=2.
	\end{equation}

	If $k={n}/{2}$, applying Lemma \ref{le12}  we have
	\begin{equation}\label{eqhhh326}
	e(\V,\V_2)\le |\V\backslash \{v\}|= n-1.
	\end{equation}
	It follows that $$e(\V,\V_1)=e(D)-e(\V,\V_2)\ge \frac{n^2}{4}+1.$$ The   pigeonhole principle ensures that there exists a vertex $u\in \V_1$ such that $d^-(u)\ge k+1$, which contradicts our assumption that $k$ is larger than or equal to the maximum indegree of $D$.
	Hence we get (\ref{eq4.12}).

	From (\ref{eq4.1}) we have
	\begin{equation}\label{eq4.13}
	e(\V_2,\V)=\frac{n^2}{4}-1,\quad e(\V_1,\V\backslash \{v\})=n-1,
	\end{equation}
	which implies that all vertices in $\V_2$ have outdegree $k$, $\V_1\rightarrow \V$, and
	$$e(\V_1,\V_2\backslash \{v\})=\frac{n}{2}-2.$$  Since $\tau=2$, we may assume
	$$N^+(v)\cap N^-(v)=\{v_1,v_2\}.$$
By Lemma \ref{le2}, each vertex in $\V_2\setminus \{v\}$ has at most one successor from $\{v_1,v_2\}$ and has a successor from $\V_2$.  Hence,
$$\V_3=\{v\}~{\rm and~} \V_4=\V_2\setminus \{v\}. $$
	Applying Lemma \ref{le12}, there exists  exactly one vertex in $\V_1$, say $y$, without a successor from $\V_2$. Moreover, by Lemma \ref{le2},
	\begin{equation}\label{eqh4.14}
	\V_1\backslash \{v_1,v_2,y\} {\rm ~matches ~}\V_2\backslash \{v\}.
	\end{equation}

	For any vertex $u_1\in \V_4$, it has a successor $u_2\in \V_2$.  Since $d^+(u_1)=k$, we have either
	\begin{equation*}
	u_1\rightarrow v_1,~~u_1\nrightarrow v_2
	\end{equation*}
	or
	\begin{equation*}
	u_1\nrightarrow v_1,~~u_1\rightarrow v_2.
	\end{equation*}
	Without loss of generality, we assume the former case holds. Applying Lemma \ref{le10}, we obtain $$v_2\rightarrow \V_1\backslash \{v_2\}. $$
Moreover, we have
	\begin{equation}\label{eq4.14}
	N^+_{\V_1}(u)=\V_1\backslash \{v_2\} {\rm ~ for~ all~  }u\in \V_4.
	\end{equation}
	Otherwise, there exists a vertex $u_3\in \V_4$ such that $$u_3\nrightarrow v_1{\rm~and ~}  u_3\rightarrow v_2.$$ Applying  Lemma \ref{le10} again we have  $$v_1\rightarrow \V_1\backslash \{v_1\},$$ which contradicts Lemma \ref{le2}.

	Since   $\V_1\rightarrow \V_1$,   Lemma \ref{le2} implies that each vertex in $\V_1$ has exactly one predecessor from $\V_1$. Then $y$ is the predecessor of $v_2$. In fact, if a vertex $v_3\in \V_1\setminus\{y\}$ is the  predecessor of $v_2$, then $v_3$ has a successor $u_4\in \V_2$. Since
	$$e(u_4,\V_1)\ge k-1 {\rm ~and~} e(v_2,\V_1)=k-1,$$
	applying Lemma \ref{le7} we have $D\notin EX(n)$, a contradiction.
	Therefore,
	\begin{equation}\label{eq4.16}
	D(\V_1)=S(y,v_2).
	\end{equation}

	From (\ref{eq4.14}) we deduce that $u_2\ne v$. Otherwise
	we have
	$$u_1\rightarrow v\rightarrow v_2{\rm ~and~} u_1\rightarrow y\rightarrow v_2,$$
	a contradiction.

	Now we assert that $u_2$ has no successor from  $\V_2\setminus\{u_1\}$.  Otherwise suppose $u_4\in \V_2\setminus\{u_1\}$ is a successor of $u_2$. Then we have $$u_1\rightarrow u_2\rightarrow u_4{ \rm ~and~}  u_1\rightarrow \V_1\backslash \{v_2\}\rightarrow u_4,$$a contradiction.

	Therefore, $u_2\rightarrow u_1$ and $u_1\leftrightarrow u_2$ is a isolated 2-cycle in $D(\V_2)$. Since $u_1$ is arbitrarily chosen in $\V_4$,  it follows that $D(\V_4)$ is the disjoint union of 2-cycles, which means $|\V_2|=n/2-1$ is odd.

	If $n/2$ is even,
	combining (\ref{eqh3.1}), (\ref{eqh4.14}), (\ref{eq4.14}) and (\ref{eq4.16}) we deduce that
	$$ex(n)=e(D)=\frac{n^2+4n}{4}$$ and $D$ is an isomorphism of $D_3$.

	If $n/2$ is odd, then from the above arguments and (\ref{eqh3.1}) we have
	\begin{equation}\label{eq4.17}
	ex(n)=e(D)=\frac{n^2+4n}{4}-1.
	\end{equation}
Again, by (\ref{eq4.1}) we have
	$$(\frac{n}{2}-k)^2\le \tau.$$
	Since $\tau\le 2$, by  Lemma \ref{le13} we obtain
	$$k=\frac{n}{2} {\it ~or~} \frac{n}{2}+1.$$

	Suppose $k= \frac{n}{2}$. By Lemma \ref{le12}, we have (\ref{eqhhh326}) .
	It follows that
	\begin{equation}\label{eqh4.17}
	e(\V,\V_1)=e(D)-e(\V,\V_2)\ge \frac{n^2}{4}.
	\end{equation}
	Recall that
	  \begin{equation*}\label{eq4.18}
	d^-(u)\le k~{\rm for~ all~} u\in \V.
	\end{equation*}
We obtain
	\begin{equation}\label{eqh4.19}
	d^-(u)= k {\rm ~for ~all~} u\in \V_1
	\end{equation} and
	$$e(\V,\V_1)=\frac{n^2}{4},~~e(\V,\V_2)= n-1.$$
	By Lemma \ref{le12}, each vertex in $\V\backslash \{v\}$ has exactly one successor in $\V_2$.
	
By (\ref{eq4.1}) and (\ref{eq4.17}) we have
	\begin{equation}\label{eq4.20}
	e(\V_2,\V)\ge k(n-k)-2,
	\end{equation}
which implies there exist at least $k-2$ vertices in $\V_2$ which have outdegree $k$.

Let  $t_1\in \V_2\backslash \{v\}$ such that $d^+(t_1)=k$ and $t_1$ has a successor $t_2 \in \V_2$. We assert that either $ \V_1\backslash \{t_1'\}\nrightarrow t_1'$ or $t_2\nrightarrow t_1'$. Otherwise, we have  $t_1\rightarrow \V_1\backslash \{t_1'\}\rightarrow t_1'$ and $t_1\rightarrow  t_2\rightarrow t_1'$, a contradiction. By Lemma \ref{le2}, we obtain $e(\V_1,t_1')\le 1$. It follows that
	$$d^-(t_1')=e(\V_2\backslash \{t_1,t_2\},t_1')+e(\{t_1,t_2\},t_1')+e(\V_1,t_1')\le k-1,$$
which contradicts (\ref{eqh4.19}).

Therefore, we have $$k=\frac{n}{2}+1.$$
	By (\ref{eq4.1}) and (\ref{eq4.17}), we get $\tau=1$ or 2.
	Now we distinguish two cases.

	{\it Case 2.1.} $\tau=1$. (\ref{eq4.1}) and (\ref{eq4.17}) lead to \begin{equation}\label{eq4.26}
	e(\V_2,\V)=(n-k)k ~{\rm and~}
	e(\V_1,\V)=n
	\end{equation}
	which imply
	\begin{equation*}
	d^+(u)=k~{\rm for~all~} u\in \V_2
	\end{equation*}
	and
	\begin{equation}\label{eq4.27}
	e(\V_1,u)=1 {\rm ~for~ all~}  u\in \V.
	\end{equation}
	  Since $|\V_2|=|\V_1|-2$ and $\alpha\le 1$, by (\ref{eq4.27}), there exist  two distinct vertices  $y_1,y_2\in \V_1$ such that
	$$d^+_{\V_2}(y_1)=d^+_{\V_2}(y_2)=0$${\rm ~and ~}
	\begin{equation}\label{eq4.28}
	 \V_1\backslash \{y_1,y_2\}{\rm~ matches ~}\V_2.
	\end{equation}
	\par
	 {\it Subcase 2.1.1}. $\V_3=\V_2$ .  Then
	\begin{equation}\label{eq4.29}
	N^+(u)=\V_1~{\rm for ~all~} u\in \V_2.
	\end{equation}
	
	We will use  the following claim repeatedly.
	\\

	{\it{\bf Claim 2.} For every $u_1\rightarrow u_2\rightarrow u_3$ in $D(\V_1)$ with $u_1\ne u_3$, we have $u_1\in \{y_1,y_2\}$. }

	In fact, if $u_1\notin \{y_1,y_2\}$, then it has a successor  $u_4\in \V_2$. Note that $N^+_{\V_1}(u_4)=\V_1$. We have $u_1\rightarrow u_2\rightarrow u_3$ and  $u_1\rightarrow u_4\rightarrow u_3$, which contradicts $D\in EX(n)$.
	\\

	Given $u\in \V_1$, denote by $$F(u)=\{u\}\cup N^+_{\V_1}(u)\cup N^+_{\V_1}(N^+_{\V_1}(u)).$$
	Given any vertex $u_1\in \V_1\setminus [F(y_1)\cup F(y_2)]$,  $u_1$ has a predecessor $u_2\in \V_1$, which has a predecessor $u_3\in \V_1$. Then we have $u_3\rightarrow u_2\rightarrow u_1$ in $D(\V_1)$. Since $u_1\notin F(y_1)\cup F(y_2)$, then $u_3\notin \{y_1,y_2\}$. Applying Claim 2, we obtain $u_3=u_1$. Hence, any vertex  in $\V_1\setminus[(F(y_1)\cup F(y_2))]$ belongs to a 2-cycle. By (\ref{eq4.27}), these 2-cycles are pairwise disjoint.

	If there is an arc between $y_1$ and $y_2$, say $y_1\rightarrow y_2$, then $y_2\rightarrow y_1$. Otherwise, $y_1$ has a predecessor $y_3\in \V_1\setminus\{y_1,y_2\}$, a contradiction with Claim 2.  By (\ref{eq4.27}), we see that
	$$D(F(y_1)\cup F(y_2))=T(y_1,y_2) $$
	and
	$D(\V_1)=T(y_1,y_2)$ or it is the disjoint union of $T(y_1,y_2)$ and some 2-cycles. Combining this with (\ref{eq4.28}) and (\ref{eq4.29}),  we obtain that $D$ is an isomorphism of $D_4$.

	Now suppose $y_1\nrightarrow y_2$ and  $y_2\nrightarrow y_1$. Let the  predecessors of $y_1$ and $y_2$ in $\V_1$ be  $y_1^*$  and  $y_2^*$, respectively.
	
	Suppose $y_1^*\ne y_2^*$. By (\ref{eq4.27})  and Claim 2, $y_2^*$ has a predecessor $t\in \{y_1, y_2\}$.   If $y_1\rightarrow y_2^*$, then $y_1^*\rightarrow y_2\rightarrow y_2^*$, which contradicts Claim 2. Hence, $y_1\leftrightarrow y_1^*$. Similarly, we have  $y_2\leftrightarrow y_2^*$.

	We assert that $y_1$ has only one successor in $V_1$. Otherwise, there exists $y_3\in \V_1\setminus\{y_1^*\}$ such that $y_1\rightarrow y_3$. Then we have $y_1^*\rightarrow y_1\rightarrow y_3$, which contradicts Claim 2. Hence, $e(y_1,\V_1)=1$. Similarly, we have $e(y_2,\V_1)=1$.
	
	Moreover, applying Claim 2 we have
	$$e(x,\V_1)=0 ~{\rm for~all~}x\in N^+(y_1^*)\cup N^+(y_2^*)\setminus\{y_1,y_2\}.$$
	
	Therefore, by (\ref{eq4.27}), $D(\V_1)$ is the disjoint union of $S(y_1,y_1^*)$,  $S(y_2,y_2^*)$, and some $2$-cycles, where $S(y_1,y_1^*)$,  $S(y_2,y_2^*)$ must appear and the 2-cycles may vanish. Combining this with (\ref{eq4.28}) and (\ref{eq4.29}),  we see that $D$ is an isomorphism of $D_5$.\\

	Suppose $y_1^*= y_2^*$. Since $y_1^*$ has a predecessor from $\{y_1,y_2\}$, without loss of generality, we let $y_1\rightarrow y_1^*$. Applying the same arguments as above, we obtain $$e(y_1,\V_1)=1{\rm ~and  ~}
	 e(x,\V_1)=0 ~{\rm for~all~}x\in N^+(y_1^*)\setminus\{y_1\}.$$
Therefore, $D$ is an isomorphism of  $D_{10}$.

	{\it Subcase 2.1.2.} $\V_3\ne \V_2$. There exist $u_1,u_2\in \V_2$ such that  $u_1\rightarrow u_2$. By (\ref{eq4.27}) and Lemma \ref{le10}, we have $$u_1'=u_2', {\rm ~ ~} u_1'\rightarrow \V_1\backslash \{u_1'\}$$  and
	\begin{equation}\label{eq4.30}
	N^+_{\V_1}(u)=\V_1\backslash \{u_1'\} {\rm ~for~ all~} u\in \V_4.
	\end{equation}
	Moreover,
	\begin{equation}\label{eq4.31}
	u\in \V_4~{\rm ~ if~} e(\V_2, u)\ge 1.
	\end{equation}

	We assert there exists no 2-path in $D(\V_2)$. Otherwise, suppose $D(\V_2)$ contains a 2-path $t_1\rightarrow t_2\rightarrow t_3$.
	By (\ref{eq4.30}) and (\ref{eq4.31}), we have
	$$ t_1,t_2,t_3\in \V_4,{\rm ~~}t_1'=t_2'=t_3'=u_1'$$ and there exists $t_4\in \V_4$ such that $t_3\rightarrow t_4$. If $N^+_{\V_1}(t_1)\rightarrow t_3$, we have
	$$t_1\rightarrow N^+_{\V_1}(t_1)\rightarrow t_3,$$
	which is another 2-path from $t_1$ to $t_3$, a contradiction. Hence we have $N^+_{\V_1}(t_1)\nrightarrow t_3$. Now (\ref{eq4.27}) and (\ref{eq4.28})  imply $t_1'\rightarrow t_3$ and
	$$t_1'\rightarrow t_3\rightarrow t_4,~~t_1'\rightarrow \V_1\backslash \{t_1'\}\rightarrow t_4,$$
	a contradiction.
	Therefore, by (\ref{eq4.31}) we obtain that  $D(\V_4)$ is a disjoint union of  2-cycles.

	By (\ref{eq4.27}), there exists $w\in \V_1$ such that $w\rightarrow u_1'$. Suppose $w$ has a successor $u_3\in \V_2$. Since
	$$e(u_3,\V_1)\ge k-1{\rm~and~} e(u_1',\V_1)=k-1,$$ applying Lemma \ref{le7} we have $D\notin EX(n)$, a contradiction. Hence,  $w$ has no successor in $\V_2$ and it is either $y_1$ or $y_2$. Without loss of generality, we assume $w=y_1$. Then $y_1\leftrightarrow u_1'$.
	
		If  $u_1'$ has no successor in $\V_2$, i.e., $u_1'=y_2$, then $D(\V_1)=S(y_1,y_2)$. Combining this with (\ref{eq4.28}) and (\ref{eq4.30}), we see that $D$ is an isomorphism of $D_6$.

	Suppose $u_1'$ has a successor $u_1^*\in \V_2$. If $u_1^*\in \V_4$, then $u_1^*$ has a successor $g\in \V_2$.  By (\ref{eq4.28}) we have
	$$u_1'\rightarrow \V_1\backslash \{u_1'\}\rightarrow g.$$ On the other hand, we have $u_1'\rightarrow u_1^*\rightarrow g$, a contradiction.
	Hence,  $u_1^*\in \V_3$. Then $D(\V_1)=S(y_1,u_1')$ and $u_1'$ has a successor $u_1^*\in \V_3$.  Thus $D$ is an isomorphism of $D_7$.
	\\

	{\it Case 2.2.} $\tau=2$.  If $D$ contains a vertex $z$ such that  $d^+(z)=k$ and $\tau(z)\le 1$, then replacing the role of $v$ by $z$ and repeating the above arguments, we can deduce that $D$ is an isomorphism of $D_4$, $D_5$, $D_6$, $D_7$ or $D_{10}$.

	Now we assume
	$$\tau(z)=2~{\rm ~for~all~}z\in \V ~{\rm such~that~} d^+(z)=k$$ and
	suppose $v_1, v_2\in \V_1$ are the two predecessors of $v$. Then by Lemma \ref{le2} we have $\V_3=\{v\}$.  By (\ref{eq4.1}) and (\ref{eq4.17}) we get
	either
	\begin{equation}\label{eq4.32}
	e(\V_2,\V)=k(n-k)
	\end{equation}
	or
	\begin{equation}\label{eq4.33}
	e(\V_2,\V)=k(n-k)-1.
	\end{equation}

If (\ref{eq4.32}) holds, then each vertex in $ \V_2$ has outdegree $k$. Since $\tau=2$,  every vertex in $\V_4=\V_2\backslash \{v\}$ has a successor in $\V_2$. It follows that the number of arcs in $D(\V_2)$  is $|\V_2\backslash \{v\}|$, which is  odd. Hence $D(\V_2)$ contains an arc not in any 2-cycle, say, $u_1\rightarrow u_2$ and $u_2\nrightarrow u_1$ with $u_1\in \V_4$.
	Then  by Lemma \ref{le2} and Lemma \ref{le12},  we have
	$$\tau(u_1)\le e(\V_1,u_1)= 1,$$
	a contradiction with our assumption.

Now suppose that (\ref{eq4.33}) holds.   Then $n/2-2$ vertices of $\V_2$ have outdegree $k$ and a vertex $x\in \V_2$ has outdegree $k-1$. By  (\ref{eq4.1}) and (\ref{eq4.33})  we have
	$$ e(\V_1,  \V_1)= \frac{n}{2}+1{\rm ~and~}  e(\V_1,  \V_2\backslash \{v\})=\frac{n}{2}-2.$$ It follows that $\V_1\rightarrow \V_1$. Moreover, $\alpha\le 1$ implies there exists a unique vertex  $y$ in $ \V_1$ without any successor in $\V_2$. By Lemma \ref{le2},
	$$\V_1\backslash \{v_1,v_2,y\}{\rm ~ matches~} \V_2\backslash \{v\}.$$

	For an arbitrary vertex $u\in \V_2\backslash \{v,x\}$, by Lemma \ref{le9} and Lemma \ref{le12}, $u$ has a unique successor $w\in \V_2$. Then $\tau(u)=2$ leads to  $u\leftrightarrow w$.  Therefore, there exist at least $ (n-6)/{4}$ 2-cycles in $D(\V_2)$. By Lemma \ref{le12}, these 2-cycles are pairwise disjoint and $D(\V_2\backslash \{v,x\})$ is the disjoint union of  $ (n-6)/{4}$ 2-cycles.

	Take any arc $u_1\rightarrow u_2$ in $D(\V_2\backslash \{x\})$. Then $u_1$ has exactly one successor in $\{v_1,v_2\}$, say, $v_1$.  Since $\V_1\rightarrow \V_1$, applying Lemma \ref{le10} we have
	$$N^+_{\V_1}(u_2)=\V_1\backslash \{v_2\}{\rm~ and~} v_2\rightarrow \V_1\backslash \{v_2\}. $$Moreover,
	\begin{equation*}
	N^+_{\V_1}(u)=\V_1\backslash \{v_2\} {\rm~for~ all~} u\in \V_2\backslash \{v,x\}.
	\end{equation*}

	Note that $v_2$ has a predecessor $v_3\in \V_1$. If $v_3$ has a successor $v_4\in \V_2$, then   $e(v_4,\V_1)\ge k-1$ and $e(v_2,\V_1)=k-1$.  By Lemma \ref{le7} we obtain $D\notin EX(n)$, a contradiction. Hence, $v_3=y$, i.e., $y\rightarrow v_2$.
	Therefore,
	$$D(\V_1)=S(y,v_2).$$

	Notice $d^+(x)=k-1$. If $N^+(x)\subset \V_1$, then by Lemma \ref{le2} we have
	$$N^+(x)=\V_1\backslash \{v_1\} {\rm ~ or~} N^+(x)=\V_1\backslash \{v_2\}. $$ Therefore, $D$ is an isomorphism of $D_8$.

	Now suppose $N^+(x)\nsubseteq \V_1$, i.e., $x$ has a successor $x_1\in \V_2$.  If $x_1= v$, then $x\nrightarrow y$. Otherwise we have
	$$x\rightarrow y\rightarrow v_2~{\rm and~}x\rightarrow v\rightarrow v_2,$$ a contradiction. Similarly, we have $x\nrightarrow v_2$, since otherwise we have
	$$x\rightarrow v_2\rightarrow z ~{\rm and }~x\rightarrow v\rightarrow z~{\rm for~all~ }z\in \V_1\setminus\{v_2\}.$$
	Therefore, we get
	$$N^+(x)=\{v\}\cup \V_1\setminus\{v_2,y\}.$$
	Take place the role of $v$ by $v_2$, we get
	$$ \mathcal{V}_1(v_2)=\{v\}\cup \V_1\setminus\{v_2\} ~and~ \mathcal{V}_2(v_2)=\{v_2\}\cup \V_2\setminus\{v\}.$$
	Moreover, we have
	$$D(\V(v_2))=S(v_1,v), N^+(x)=\V_1(v_2)\setminus\{y\}$$
	and $\V_1(v_2)\setminus\{v_1,v\}$ matches $\V_2(v_2)$ with $y\rightarrow v_2$.
	 Thus $D$ is an isomorphism of $D_8$.

	Finally, suppose $x_1\ne v$.  Since $D(\V_2)$ contains $({n-6})/{4}$ pairwise disjoint 2-cycles, we have $x_2\in \V_2\setminus\{v,x\}$ such that $x_1\leftrightarrow x_2$. Moreover, $x_2$ has a predecessor $x_2^*\in \V_1$, since $\V_1\rightarrow \V_2$.
	It follows that
	$$x\nrightarrow v_2~{\rm and~} x\nrightarrow x_2^*.$$
	Otherwise we have either
	$$x\rightarrow x_1\rightarrow \V_1\setminus \{v_2\}, x\rightarrow v_2\rightarrow \V_1\setminus \{v_2\}$$
	or
	$$x\rightarrow x_1\rightarrow x_2, x\rightarrow x_2^*\rightarrow x_2,$$
	a contradiction.
	Therefore, $N^+(x)=\V_1\backslash \{v_2,x_2^*\}$ and $D$ is an isomorphism of  $D_9$.
	
	Note that $$e(D_i)=ex(n) {\rm ~for~}  i=1,2,\ldots,10. $$ Applying Lemma \ref{le1}, we get the second part of Theorem 2.1.
	This complete the proof.

	\section*{Acknowledgement}

	Partial of this work was done when Huang was visiting Georgia Institute of Technology with the financial support of China Scholarship Council. He thanks China Scholarship Council and Georgia Tech for their support. He also thanks Professor Xingxing Yu for helpful discussion on graph theory during his visit. The research of Huang
	was supported by the NSFC grant 11401197, and a Fundamental Research
	Fund for the Central Universities.

\end{document}